\documentclass[11pt]{amsart}
\usepackage{amsmath}
\usepackage{amsfonts}
\usepackage{amssymb}
\newtheorem{thm}{Theorem}[section]

\newtheorem{cor}[thm]{Corollary}

\newtheorem{remark}[thm]{Remark}
\newtheorem{lemma}[thm]{Lemma}
\newtheorem{prop}[thm]{Proposition}
\newtheorem{defn}[thm]{Definition}
\newtheorem{prob}[thm]{Problem}
\newcommand{\bb}[1]{\mathbb{#1}}
\newcommand{\cl}[1]{\mathcal{#1}}

\begin{document}

\title[WEP]{Weak Expectations and the Injective Envelope}

\author[V.~I.~Paulsen]{Vern I.~Paulsen}
\address{Department of Mathematics, University of Houston,
Houston,
Texas
77204-3476,
U.S.A.}
\email{vern@math.uh.edu}

\date{\today}

\thanks{Research supported in part by a NSF grant DMS-0600191}

\begin{abstract}
Given a unital C*-subalgebra $\cl A \subseteq B(\cl H),$ we study the set of
all possible images of the injective envelope $I(\cl A)$ of $\cl A$ that are contained
in $B(\cl H)$ and their position relative to the double commutant of
the algebra in order to develop more information about the existence
or non-existence of weak 
expectations. We also introduce a new category, such that the
injective envelope of $\cl A$ in the new category is always contained
in the double commutant and study the relationship between these two
injective envelopes and the existence of weak expectations.

\end{abstract}

\maketitle

\section{Introduction and Preliminaries}

A unital $C^*$-subalgebra, $\cl A \subseteq B(\cl H)$ of the bounded 
linear operators on 
a Hilbert space, is said to have a {\em weak expectation} 
provided that
there is a completely positive map from $B(\cl H)$ the double
commutant of $\cl A,$ denoted $\cl A'',$
that is the identity on $\cl A.$
A unital $C^*$-algebra $\cl A$ is said to have the {\em weak 
expectation property(WEP)} provided that for every faithful 
*-representation $\pi : \cl A \rightarrow B(\cl H)$ of $\cl A$ 
onto a Hilbert space $\cl H$, the $C^*$-subalgebra $\pi(\cl A)$ 
has a weak expectation.
If we let $\pi_u$ denote the universal representation of $\cl 
A$, so that the double commutant $\pi_u(A)''$ is identified with the double dual, 
$A^{**}$, then it is known that $\cl A$ has the WEP if and only 
if this representation has a weak expectation.

Blackadar\cite{Bl} observed that $\cl A$ has a weak 
expectation if and only if 
$\cl A''$ contains an operator system that is
completely isometrically isomorphic to $I(\cl A)$,
the injective envelope of $\cl A$, via a map that fixes $\cl A.$
Thus, the WEP is equivalent to
the existence of a copy of $I(\cl A)$ inside $\cl A^{**}$.

However, in general, $B(\cl H)$ will contain many operator 
systems that contain $\cl A$ and are completely
isometrically isomorphic to $I(\cl A).$  For a weak expectation to 
exist we only need 
one of these copies
of the injective envelope to be contained in the double 
commutant. In general, a C*-subalgebra $\cl A \subseteq B(\cl H),$
which has a weak expectation, will also have some copies of $I(\cl A)$ that
are not contained in $\cl A''.$  We construct such an algebra below.

For this reason, given a unital C*-subalgebra $\cl A \subseteq B(\cl H),$ we
are led to a more detailed study of the collection of all 
possible ''copies" of $I(\cl A)$ that lie inside $B(\cl H)$, the 
relationships between these various copies of the injective envelopes and the collection of projections onto these copies.  One object that will play a role in our study is the intersection of all possible copies of $I(\cl A).$

Another tool that we shall use is a new type of ''envelope" of $\cl A$ that is always a subset of $\cl A'',$ but which is, generally, only injective in a sense relative to $\cl A''.$ It turns out that $\cl A$ has an weak expectation if and only if this new envelope is injective.

In section 2, we further develop some of Hamana's ideas.  We introduce and study this new type of envelope and
simultaneously obtain additional information about the set of all
projections onto copies of the injective envelope. This is the set that Hamana\cite{H1}
calls the {\em minimal $\cl A$-projections.} Section 3 applies these
ideas to the study of weak expectations. Section 4 is devoted to
developing the properties of the set that is the intersection of all
copies of the injective envelope. We prove that this set is
simultaneously a reflexive cover of $\cl A$ and a new type of order
completion of $\cl A.$ We compute this set for a few examples.

We close this section by justifying a few of the comments above. First, since the entire motivation for this study relies on Blackadar's result \cite{Bl}, we provide an independent argument which also serves as an introduction to many of Hamana's ideas. 

To obtain Blackadar's result, first assume that $\cl A \subseteq B(\cl H)$ possesses a weak expectation.  Note that
since $B(\cl H)$ is injective, the identity map on $\cl A$ extends to
a map of $I(\cl A)$ into $B(\cl H).$ Composing this latter map with
the weak expectation yields a completely positive map
of $I(\cl A)$ into $\cl A''$ that 
fixes $\cl A$. The map of $I(\cl A)$ into $\cl A''$ must be a complete isometry on the injective 
envelope, by the fact \cite{H1} that the
injective envelope is an essential extension of $\cl A$. Conversely, if $I(\cl A)$ can be embedded completely isometrically isomorphically into $\cl A'',$ then since $I(\cl A)$ is injective, we may extend the identity map on $\cl A$ to a completely positive map of $B(\cl H)$ to $I(\cl A).$ Composing this extension with the inclusion of  $I(\cl A)$ into $\cl A''$ yields the desired weak expectation.

Next we would like to point out that the fact that the image of one representation $\pi: \cl A \to B(\cl H)$ has a weak expectation is not enough to guarantee that $\cl A$ has the weak expectation property. In fact, every C*-algebra has at least one representation that has a weak expectation. The double commutant of the reduced atomic representation is always injective 
and hence this representation possesses a weak
expectation. 

Next we give an example to show that it is possible for a C*-algebra to have a weak expectation, while some copies of the injective envelope are not contained inside the double commutant of the algebra.
To this end, let $\cl A$ be a C*-algebra with the WEP that is not injective and take the universal representation of the C*-algebra, $I(\cl A).$  Then $\pi_u(I(\cl A))$ is a copy of the injective envelope of $\pi_u(\cl A),$ but we claim that this copy of of the injective envelope is not contained in $\pi_u(\cl A)''.$  If not, then we would have that $\pi_u(I(\cl
A)) \subseteq \pi_u(\cl A)''.$ But this implies that $I(\cl A)^{**} \simeq \pi_u(I(\cl
A))'' \subseteq \pi_u(\cl A)'' \simeq \cl A^{**}$ and this inclusion of $I(\cl A)^{**} \subseteq \cl A^{**}$ is weak*-continuous and the identity on $\cl A.$ From this it follows that $\cl A = I(\cl A),$ contradicting the choice of $\cl A.$ Since $\cl A$ had the WEP, $\pi_u(\cl A)$ is a
C*-algebra with a weak expectation and so some copy of its injective
envelope is contained in $\pi_u(\cl A)'',$ but $\pi_u(I(\cl A)) \simeq
I(\pi_u(\cl A))$ is a copy of its injective envelope that is not
contained in the double commutant.

Finally, although our main interest is in the case of C*-algebras, many of the ideas extend to the case of operator spaces.

The definition of the injective envelope of a C*-algebra has been extended to general operator spaces by Ruan\cite{Ru} and has been shown to enjoy similar properties to those proven by Hamana, namely, that it is a rigid, essential and injective extension.
Pisier\cite{Pi} defined an operator space $V$ to have the WEP provided
that the identity on $V$ extends to a completely contractive map of
Ruan's injective envelope $I(V)$ into $V^{**}.$ Thus, a finite
dimensional operator space has the WEP if and only if it is
injective. We prove below that Pisier's definition, as in the case of C*-algebras, is equivalent to one involving weak expectations.

\begin{defn}
Let $V \subseteq B(\cl H)$ be an operator space, we say that $V$ has a weak expectation, if
there is a completely contractive map from $B(\cl H)$ into the
weak*-closure of $V,$ that is the identity on $V.$ 
We define an operator space $V$ to have the weak expectation property(WEP) provided that there 
exists a
completely isometric embedding of the injective envelope of $V$, 
$I(V)$ into $V^{**}$
that is the identity on $V$.
\end{defn}

The equivalence of the WEP to every inclusion possessing a weak
expectation is a little more subtle for operator spaces, so we prove this below.

\begin{prop} Let $V$ be an operator space. Then $V$ has the WEP if and
  only if for every Hilbert space, $\cl H,$ and for every complete
  isometry, $\varphi:V \to B(\cl H),$ the subspace $\varphi(V)
  \subseteq B(\cl H)$ has a weak expectation.
\end{prop}
\begin{proof} Assume that $V$ has the WEP and let $\varphi:V \to B(\cl
  H)$ be a complete isometry. Let $W$ denote the weak*-closure of
  $\varphi(V)$, so that by \cite[Lemma~1.4.6 and 1.4.8]{BLeM} $W$ has
  a pre-dual and $\varphi$ extends to a complete contraction of
  $V^{**}$ into $W$. Composing this map with the embedding of $I(V)$
  into $V^{**}$ yields a completely contractive mapping of $I(V)$ into
  $W$ that is restricts to be a complete isometry on $V$. Since $I(V)$
  is an essential extension of $V$, this map must also be a complete
  isometry on $I(V).$

Conversely, assume that every completely isometric embedding of $V$
into $B(\cl H)$ possesses a weak expectation. By
\cite[Lemma~1.4.7]{BLeM}, there exists a weak*-continuous completely
isometric embedding of
$V^{**}$ onto a weak*-closed subspace of $B(\cl H),$ for some $\cl
H$. Taking any weak expectation for this embedding and composing it
with any completely isometric inclusion of $I(V)$ into $B(\cl H),$
yields the desired embedding of $I(V)$ into $V^{**}.$
\end{proof}

We remark that the above proof applied to the case of C*-algebras,
shows the earlier equivalence of the two characterizations of
WEP. 

\section{Minimal Projections and W-injectivity}

In this section we take a closer look at the ideas contained in 
Hamana's constructions \cite{H1, H2} of the injective envelope and prove a 
number of facts that are consequences of these ideas, but that 
seem to have not been remarked earlier.
In order to better understand the weak expectations it is useful to examine the
extent that Hamana's constructions can be carried out in the setting 
where one
has operator spaces $V \subseteq W$ with $W$ a dual operator space
that is not necessarily injective.

\begin{defn} Let $W$ be an operator space. We let $\frak W$ denote the 
category whose
objects are operator subspaces of $W$ and given two operator subspaces 
$X,Y$ of $W$ we let the maps
from $X$ to $Y$, denoted $M(X,Y)$, be the set of completely 
contractive maps from $\phi : W \to W$ 
 such that $\phi(X) \subseteq Y$.
 We define an object $Y$ to be {\bf injective in $\frak W$} or more 
shortly,
 {\bf $W$-injective} provided that whenever $X_1 \subset X_2$ are 
operator subspaces of $W$ and
 $\phi_1 \in M(X_1, Y)$, then there exists $\phi_2 \in M(X_2,Y)$ such 
that the restriction
 of $\phi_2$ to $X_1$ is $\phi_1$.
We say that $X$ and $Y$ are {\bf $W$-isomorphic} provided that there 
exists $\phi \in M(X,Y)$ and $\psi \in M(Y,X)$ such that $\psi \circ 
\phi$ is the identity on $X$ and $\phi \circ \psi$ is the identity on 
$Y$.
 \end{defn}
 
 It is not hard to see that $Y$ is $W$-injective if and only if there 
exists a completely
contractive idempotent from $W$ onto $Y$. So, for example, $W$ is 
always injective
in $\frak W$, even though it need not be an injective operator space 
in the traditional sense.
If $W$ is an injective operator space, in the usual sense, then it 
follows that every $W$-injective
operator space is also injective in the usual sense.
If $X$ and $Y$ are $W$-isomorphic, then they are completely 
isometrically isomorphic, but the converse is not apparent.

We now show that when $W$ is a dual space, then many of Hamana's 
results about minimal
projections and injective envelopes hold and allow one to construct an 
injective envelope in $\frak W$
with analogous properties to the usual injective envelope.

We begin with the relevant definitions. 

\begin{defn} Let $V \subseteq W$ be operator spaces. We call a 
completely
contractive map $\phi: W \to W$ a {\bf $V$-map} provided that $\phi(v) 
=v$ for
every $v \in V$. We say that $X$ is a {\bf
$W$-essential extension of $V$} provided that $V \subseteq X \subseteq 
W$ and
whenever $\phi: W \to W$ is a $V$-map, then there exists a $V$-map 
$\psi$ such that
$\psi \circ \phi(x) = x$ for every $x \in X$.
We say that $X$ is an {\bf $W$-rigid extension of $V$} provided that 
$V \subseteq X \subseteq W$
and the only $V$-map in $M(X,X)$ is the identity map on $X$.
\end{defn}

When $W$ is injective in the usual sense, then $X$ is a 
$W$-rigid(respectively, $W$-essential) extension of $V$ if and only if 
it is a rigid(respectively, essential) extension of $V$ in the usual 
sense, that is, if and only if the only completely contractive map 
from $X$ to itself that fixes $V$ is the identity on $X$(respectively, 
any completely contractive map on $X$ that is a complete isometry on 
$V$ is a complete isometry on $X$).

Given a $V$-map $\phi: W \to W$ it induces a seminorm on $W$ by 
setting $p_{\phi}(w) = \|\phi(w)\|.$
The set of such seminorms is partially ordered by $p_{\phi} \le 
p_{\psi}$ if and only if $p{\phi}(w) \le p_{\psi}(w)$ for every $w \in 
W.$

Now assume that $W$ is a dual Banach space, so that $W$ is endowed 
with a
weak*-topology. Given any chain of such seminorms, then by taking a 
point-weak*-limit point of the
set of maps, we obtain a new $V$-map $\psi$ such that $p_{\psi}$ is a 
lower bound for the
chain. Thus, by Zorn's lemma, there exist minimal such seminorms.

To use the above argument to prove the existence of minimal seminorms, 
one only needs that $M(W,W)$ is endowed with a topology that
makes it a compact set and such that for every $w \in W$ the map, 
$\phi \to \|\phi(w)\|$ is continuous.

\begin{defn} We call an operator space $W$  {\bf admissable} provided 
that for every
operator subspace $V$ and every $V$-map $\phi$ there exists a $V$-map 
$\psi$ such that
$p_{\psi} \le p_{\phi}$ with $p_{\psi}$ a minimal seminorm.
\end{defn}

Thus, by our above remarks every operator space $W$ that has an 
appropriate topology on $M(W,W)$ is
admissable and in particular every operator space that is a dual 
Banach space is admissable.

\begin{thm}\label{thm1} Let $V \subseteq W$ be operator spaces with 
$W$ admissable and let $\phi: W \to W$ 
be a $V$-map such that the seminorm $p_{\phi}$ is minimal among this 
family of seminorms,
then $\phi$ is a completely contractive projection and the range of 
$\phi, \cl R(\phi)$ is a $W$-rigid, $W$-injective
extension of $V$.
\end{thm}
\begin{proof}
Set $\phi^{(2)} = \phi \circ \phi$ and inductively, $\phi^{(n+1)} = 
\phi \circ \phi^{(n)}$.
Define $\psi_n = \frac{\phi + \cdots + \phi^{(n)}}{n}$. Since 
$p_{\psi_n} \le p_{\phi}$ we have equality
of these seminorms. But $\psi_n(x - \phi(x)) = \frac{\phi(x) - 
\phi^{(n+1)}(x)}{n}$ which tends to 0
in norm. Hence, $0 = \phi(x - \phi(x))$ and so it follows that $\phi$ 
is idempotent and so the range
of $\phi$ is $W$-injective.

Let $Y = \cl R(\phi)$, it remains to show that $Y$ is a $W$-rigid 
extension of $V$. 
Let $\psi \in M(Y,Y)$ be a
$V$-map. Since $p_{\psi \circ \phi} \le p_{\phi}$ we again have 
equality. Thus, by the above
$\psi \circ \phi$ must be an idempotent map. Let $y \in Y$, then
$$\| y - \psi(y)\| = \|\phi(y - \psi \circ \phi(y))\| = \| \psi \circ 
\phi(y - \psi \circ \phi(y)\|
=0.$$ Hence, $\psi$ is the identity on $Y$ and so $Y$ is a $W$-rigid 
extension of $V$.
\end{proof}

For all of the following results, we assume that $W$ is an admissable 
operator space and that
$V \subset W$.

\begin{lemma}\label{lemma1} Let $W$ be an admissable operator space, 
with $V \subset W$.
Let $Y$ be a $W$-rigid and $W$-injective extension of $V$. Let $E: W 
\to Y$ be a 
$V$-map and let $\psi$ be any $V$-map. If 
$p_{\psi} \le p_E$, then $E \circ \psi = E$ and $ker(\psi) = 
ker(E).$
\end{lemma}
\begin{proof}
Since $Y$ is a $W$-rigid extension of $V$, we have that $E(y)=y$ for every $y \in Y,$ and hence, $E \circ E= E.$ For $y \in Y$ we have that $E \circ \psi(y) = y$ by 
rigidity. For $k \in ker(E)$ we have that $\|\psi(k)\| \le 
\|E(k)\| = 0 $ and so $ker(E) \subseteq ker(\psi)$. Since every 
element $x \in W$ can be written as $x= y+k$ for $y \in Y$ 
and $k \in ker(E)$, we have that $E \circ \psi(x) = y = 
E(x)$ and the first claim follows.

If $x \in ker(\psi)$, then $E(x) = E \circ \psi(x) = 0$ and so 
$ker(\psi) \subseteq ker(E).$
\end{proof}

\begin{prop}\label{propn1}
Let $Y$ be a $W$-rigid and $W$-injective extension of $V$ and let $E: 
W \to Y$ be a
 $V$-map, then $p_E$ is a minimal $V$-seminorm.
\end{prop}
\begin{proof}
Suppose not. Then we may choose an $V$-map $\phi$ such that 
$p_{\phi} \le p_E$ with $\|\phi(x)\| < \|E(x)\|$ for some $x$.
But by the above, $\|E(x)\| = \|E \circ \phi(x)\| \le 
\|\phi(x)\|$, a contradiction. 
\end{proof}

\begin{prop}\label{propn2}
Let $V \subseteq W$ with $W$ admissable and let $Y$ be a $W$-rigid and $W$-injective extension of $V$. Let $E: W 
\to Y$ be a 
$V$-map and let $\phi$ be any $V$-map. Then
$ker(\phi \circ E) = ker(E), E \circ \phi \circ E = E$ and
$\phi \circ E$ and $E \circ \phi$ are 
completely contractive projections onto $W$-rigid and $W$-injective 
extensions of $V$. 
\end{prop}
\begin{proof}
Since $\phi$ is a contraction, $p_{\phi \circ E} \le p_E$
and so we may apply Lemma~2.5 with $\phi \circ E =\psi$ to obtain
that $E \circ \phi \circ E = E$ and that $ker(\phi \circ E)= ker(E).$
Hence, $(\phi \circ E) \circ (\phi \circ E)= \phi \circ E$ and $(E \circ \phi) \circ (E \circ \phi) = E \circ \phi$ and so these maps are completely contractive projections as claimed.

By Proposition~2.6, $p_E$ is a 
minimal
$V$-seminorm and hence $p_{\phi \circ E}$ is also a minimal 
$V$-seminorm
and hence by Theorem~2.4, $\cl R(\phi \circ E)$ is a $W$-rigid and $W$-injective extension of $V$. Finally, since $E \circ \phi \circ E = E,$ we have that $\cl R(E \circ \phi) = \cl R(E)= Y,$ which is a $W$-rigid and $W$-injective extension of $V$.
Finally, since $E \circ \phi \circ E = E$, we have that 
$\cl R(E \circ \phi) =
\cl R(E)$ and so $E \circ \phi$ is a projection onto $\cl R(E)$.
\end{proof}

\begin{cor}\label{cor1}
Let $V \subseteq W$ with $W$ admissable and let $\phi: W \to W$ be a $V$-map whose range is contained in a  
$W$-injective, $W$-rigid extension $Y$ of $V$,
then $\phi$ is a projection onto $Y$.
\end{cor}
\begin{proof}
Let $E$ be a projection onto $Y$, then $E \circ \phi = \phi$.
By the above result $E \circ \phi$ is a projection onto an
$W$-injective subspace of $Y$ which by the $W$-rigidity of $Y$, must 
be all of $Y$.
\end{proof}

Recall that Hamana\cite{H1} introduces a partial order on projections 
by
defining $E \preceq F$ if and only if $E \circ F  = F \circ E = E.$
Note that this is equivalent to requiring that $E \circ F \circ E = F 
\circ E \circ F = E$.
To see this note that the first set of equalities clearly implies the 
second set.
If the second set of equalities holds then $F \circ E =F \circ (E 
\circ F \circ E) = 
(F \circ E \circ F) \circ E = E^2 = E$ and similarly,
$E \circ F = E$. 

\begin{thm}\label{thm2}
Let $V \subseteq W$ with $W$ admissable and let $E: W \to W$ be a $V$-map. Then the following are equivalent:
\begin{itemize}
\item[i)] $p_E$ is a minimal $V$-seminorm,
\item[ii)] $E$ is a projection onto a $W$-injective, $W$-rigid 
extension of $V$,
\item[iii)] $E$ is  minimal in the partial order on $V$-projections.
\end{itemize}
Moreover, if $E_1, E_2$ are two such minimal $V$-projections, then $\cl R(E_1)$ and $\cl R(E_2)$ are $W$-isomorphic.
\end{thm}
\begin{proof}
The proof that i) implies ii) is Theorem~\ref{thm1} and
that ii) implies i) is
Proposition~\ref{propn1}.
Assume ii) and let $F \preceq E$, then by Corollary~\ref{cor1}, $F$ is also a 
projection
onto $Y$. Hence, $E= F \circ E = F$ and so E is minimal.

Next, assuming iii), let $p_F \le p_E$ be a minimal $V$-seminorm. 
Then $p_{E \circ F}$ is also a minimal $V$-seminorm and so $F$ and $E 
\circ F$ are projections
 onto  $W$-injective, $W$-rigid extensions of $V$. 
Hence by another application
of Corollary~\ref{cor1}, $E \circ F \circ E$ is another projection onto 
a $W$-injective, $W$-rigid extension of $V$.
But $E \circ(E \circ F \circ E)  = E \circ F \circ E = (E 
\circ F \circ E) \circ E$ and hence $E = E \circ F \circ E$. Thus, 
$E$ is a projection onto a $W$-injective, $W$-rigid extension of $V$.

Finally, if $E_1$ and $E_2$ are minimal $V$-projections, with ranges $Y_1$ and 
$Y_2$,
respectively, then $E_2 \circ E_1$ defines a completely isometric 
$W$-isomorphism of
$Y_1$ onto $Y_2$, with inverse $E_1 \circ E_2$.
\end{proof}

Thus, we see that all $W$-injective, $W$-rigid extensions of $V$ are 
$W$-isomorphic, provided $W$
is admissable.

\begin{defn} Let $W$ be an admissable operator space and $V$ a 
subspace. We call any $V$-map, $E:W \to W$ that satisfies the
equivalent properties of Theorem~2.9 a {\bf minimal
$V$-projection}(with respect to $W$) and let $\cl E_W(V)$ denote the
set of all minimal $V$-projections. We
let {\bf $I_W(V)$} denote the $W$-isomorphism class of the range of 
a minimal $V$-projection and
we call this operator space the {\bf $W$-injective envelope of 
$V$}.
Any operator subspace of $W$ that is the range of a minimal 
$V$-projection will be called
a {\bf copy of $I_W(V)$}.
\end{defn}

When $V = \cl A$ is a unital $C^*$-algebra and $W= \cl A'',$ then any minimal $\cl A$-projection $E$ is also a unital, completely positive map and, hence, $I_{\cl A''}(\cl A)$ will be a $\cl A''$-injective  $C^*$-algebra when endowed with the Choi-Effros product, $E(x) \circ E(y) = E(E(x)E(y)).$

When $W$ is injective, then $I_W(V)= I(V),$ the usual injective
envelope.

The following result shows that, in general, $\cl E_W(V)$ can be quite
large and explains some of its algebraic structure. Note that the set of $V$-maps, which we shall denote by $\cl M_W(V),$ is a semigroup under composition
with the identity map on $W$ serving as an identity.

\begin{prop} Let $W$ be an admissable operator space and let $V$ be an
  operator subspace. Then $\cl E_W(V)$ is the unique minimal, non-empty, two sided ideal in the semigroup $\cl M_W(V).$
\end{prop}
\begin{proof} 
Given any $V$-map $\phi$ and $E \in \cl E_W(V),$
  we have that $\phi \circ E \in \cl E_W(V)$ and $E \circ \phi \in \cl
  E_W(V),$ by Proposition~2.7 and Theorem~2.9. Thus, $\cl E_W(V)$ is a two sided ideal in $\cl M_W(V).$

Given any non-empty two sided ideal $\cl J$ in $\cl M_W(V),$ let $\phi \in \cl J,$ and let $E \in \cl E_W(V).$ Then by Proposition~2.7, $E = E \circ \phi \circ E \in \cl J,$ and hence, $\cl E_W(V) \subseteq \cl J.$
\end{proof}

The following result identifies the minimal left ideals. Given any $\phi \in \cl M_W(V),$ we let $\cl L_{\phi}= \{ \psi \circ \phi : \psi \in \cl M_W(V) \},$ denote the left ideal generated by $\phi.$

\begin{thm} Let $W$ be an admissable operator space, let $V$ be an
  operator subspace, and let $E \in \cl E_W(V).$  Then $\cl L_E =
  \{ F \in \cl E_W(V): p_F=p_E \} = \{ F \in \cl
  E_W(V): ker(F) =ker(E) \},$  $\cl L_E$ is a convex set and a minimal non-empty left
  ideal in the semigroup of $V$-maps. Moreover, every minimal, non-empty left ideal in $\cl M_W(V)$ is equal to $\cl L_E$ for some $E \in \cl E_W(V).$
Finally, if $F_1, F_2 \in \cl L_E,$ then 
$F_1 +F_2 = F_1 \circ F_2 +  F_2 \circ F_1.$
\end{thm}
\begin{proof} First we show that the three sets are equal. If $\phi \circ E \in \cl L_E,$ then $\|\phi \circ E(w)\| \le \|E(w)\|,$ so by minimality, $p_{\phi \circ E} = p_E,$ so the first set is contained in the second. If $p_F=p_E,$ then, clearly, $ker(F)= ker(E),$ so the second set is contained in the third. If
  $ker(F) = ker(E),$ then since $E(w - E(w))=0,$ we have that $F(w) =
  F \circ E(w),$ so that $F = F \circ E \in \cl L_E,$ and all three sets are equal.

Let $F_1, F_2 \in \cl L_E$ and let $0 \le t \le 1.$  Then
  for any $w \in W,$ we have that $\|tF_1(w) + (1-t)F_2(w)\| \le
  \|E(w)\|.$
Since $p_E$ is a minimal $V$-seminorm, it follows that $p_{tF_1
  +(1-t)F_2}= p_E$ and hence, $tF_1 +(1-t)F_2 \in \cl L_E.$
Thus, $\cl L_E$ is convex.

Now, let $\cl J$ be any minimal non-empty left ideal, let $\phi \in \cl J,$ and let $F \in \cl E_W(V)$ be any element.
By Proposition, $E= F \circ \phi \in \cl E_W(V)$ and $E \in \cl J.$ Thus, by the minimality of $\cl J,$ $\cl J= \cl L_E.$

Finally, since $(F_1 +F_2)/2 \in \cl L_E,$ we have that \\
$(F_1+F_2)/2 = (F_1+ F_2)/2 \circ (F_1+F_2)/2 = 1/4(F_1 + F_1 \circ
F_2 + F_2 \circ F_1 +F_2),$\\
which implies $F_1+F_2 = F_1 \circ F_2 + F_2 \circ F_1.$
\end{proof}

\begin{cor} Let $W$ be an admissable operator space and let $V$ be a subspace.
Then $\cl E_W(V)$ is the disjoint union of the minimal left ideals in $\cl M_W(V).$
\end{cor}
\begin{proof} Every minimal non-empty left ideal is of the form $\cl L_E$ for some $E \in \cl E_W(V),$ and $\cl L_E \subseteq \cl M_W(V)$ and any two such ideals are either disjoint or equal.
\end{proof}

Similar results hold for the right ideal, $\cl R_E,$ generated by $E \in \cl E_W(V).$ We record some of them without proof.

\begin{prop} Let $W$ be an admissable operator space, let $V$ be an operator subspace and let $E \in \cl E_W(V).$  Then $\cl R_E = \{ F \in \cl E_W(V): \cl R(F)= \cl R(E) \},$ $\cl R_E$ is a convex set and a minimal right ideal. Moreover, every minimal, non-empty right ideal is equal to $\cl R_E$ for some $E \in \cl E_W(V)$ and $\cl E_W(V)$ is the disjoint union of all minimal right ideals.
\end{prop}

The following result gives a way to obtain copies of $I_W(V)$ and
clarifies the relationship of 
$I_W(V)$ with the usual injective envelope.

\begin{thm}\label{thm3} Let $V \subseteq W \subseteq B(\cl H)$ with 
$W$ admissable. If
$\phi: W \to W$ is a minimal $V$-projection, so that 
$Y= \cl R(\phi)$
is a copy of $I_W(V)$,
and there is a copy $S \subseteq B(\cl H)$ of $I(Y)$ and a completely 
contractive projection 
$E: B(\cl H) \to S$ that is an extension of $\phi$ to $B(\cl H)$ such that  
$Y = W \cap S.$
\end{thm}
\begin{proof}
Among all extensions of $\phi$ to $B(\cl H)$ choose one, say $E$ such 
that the induced seminorm on 
$B(\cl H)$
is minimal among the set of all such seminorms. 
The existence of such a minimal seminorm is guaranteed by Zorn's 
lemma, since every chain has a
lower bound given by taking a point weak*-limit point, as above. 
Setting $\psi_n = \frac{E + E \circ E + \ldots E^{(n)}}{n}$,
we have that $\psi_n$ still extends $\phi$ and produces a smaller 
seminorm on $B(\cl H)$ and 
consequently
must be equal to the seminorm induced by $E$. Apply to $x - E(x)$, as 
before, to deduce that 
$E$ is idempotent.

Hence $E$ is a projection onto some(necessarily) injective operator 
space $S$ and from this it 
follows that
$\phi$ is the projection onto $W \cap S$. 

We now prove that $S$ is a rigid extension of $Y.$ To this end suppose 
that $\gamma : S \to S$ is a completely contractive map that fixes $Y$.
Then $\gamma \circ E$ is another extension of $\phi$ with $p_{\gamma 
\circ E} \le p_E$ and hence we must have equality of these two 
seminorms. Hence $\gamma \circ E$ must also be idempotent. Arguing as 
in the last line of \ref{thm1}, we obtain that $\|s - \gamma(s)\| = 
0$ and so $\gamma$ is the identity on $S$.

Since $S$ is injective and a rigid extension of $Y$ we have that $S$ 
is completely isometrically isomorphic to $I(Y)$ via a map that fixes 
$Y$. That is, $S$ is one of the copies of $I(Y)$ in $B(\cl H)$.
\end{proof}

If $S$ is an arbitrary copy of $I(Y)$ in $B(\cl H)$ then it might not 
be the case that $W \cap S = Y$ or that the projection onto $S$ 
satisfies, $E(W) \subseteq W$, but we do not have a concrete example 
where these fail.

Note that since $V \subseteq Y$, we will have that any copy of $I(Y)$ will contain a copy of $I(V).$ Moreover, since any two copies, $Y_1, Y_2$ of $I_W(V)$ are $W$-isomorphic, they are completely isometrically isomorphic via a map that fixes $V$ and hence, any copies of $I(Y_1)$ and $I(Y_2)$ in $B(\cl H)$ will be completely isometrically isomorphic via a map that fixes $V.$ But we do not know if $I(V)=I(Y),$ or equivalently, if any completely contractive map, $\psi:I(Y) \to I(Y)$ that fixes $V$ is necessarily the identity map.

\begin{defn} Let $V \subseteq W$ with $W$ admissable. If $V \subseteq Y \subseteq W$ is any copy of $I_W(V)$, then we set $I^W(V) =I(Y)$ and recall that this operator space is uniquely determined up to a completely isometric isomorphism that fixes $V$ and is independent of $Y.$
\end{defn}

The following gives a characterization of $I^W(V)$ in the main case of
interest.

\begin{prop} Let $V \subseteq W \subseteq B(\cl H),$ with $W$
 admissable and let $E:B(\cl H) \to B(\cl H)$ be a $V$-map. If
  $p_E$ is 
minimal among all $V$-seminorms on $B(\cl H)$ such
  that $E(W) \subseteq W,$ then $E(B(\cl H))$ is a copy of $I^W(V)$
  and $E(W)$ is a copy of $I_W(V).$ 
\end{prop}
\begin{proof} If $p_E$ is minimal in the above sense,
then arguing as in the proof of Theorem~2.4, we see that $E$ is
idempotent and hence a $V$-projection. Let $\phi:W \to W$ be the
restriction of $E$ to $W.$ 

We claim that $p_{\phi}$ is a minimal $V$-seminorm on $W$. If not then
we have $\psi:W \to W,$ a $V$-map such that $p_{\psi} \le p_{\phi}.$
Then we have that $\|\psi(x- \phi(x))\| \le \|\phi(x- \phi(x))\|=0,$
and hence, $\psi(x) = \psi\circ \phi(x),$ for any $x \in W.$

Let $F:B(\cl H) \to B(\cl H)$ be any completely contractive extension
of $\psi.$ Then $p_{F \circ E} \le p_E$ and hence they are
equal. Thus, for any $x \in W,$ we have, $\|\psi(x) \| = \|\psi \circ
\phi(x)\| = \|F \circ E(x) = \|E(x)\| = \|\phi(x)\|$ and so $\phi$ is
a minimal $V$-seminorm on $W$.

Hence, $\phi(W)=E(W)$ is a copy of $I_W(V).$
Now in Theorem~2.11, it was shown that if $F$ is any map that extends
$\phi$ and has minimal seminorm on $B(\cl H)$ among all maps that
extend $\phi,$ then $F$ is a projection onto a copy of the injective
envelope of $\phi(W)$. But, by the choice of $E$, it is minimal among
all maps that extend $\phi= E \mid_W.$ Hence, the $E(B(\cl H))$ is a
copy of the injective envelope of $\phi(W)$ and hence is a copy of
$I^W(V).$
\end{proof} 

\begin{remark} In Proposition~2.13, we are not asserting that such a
  minimal $p_E$ exists, only that when it does it has the asserted
 properties.
However, if $V \subseteq W \subseteq B(\cl H)$ and $W$ is weak*-closed,
  then a $V$-map $E:B(\cl H) \to B(\cl H)$ such that $p_E$ is minimal
  among all $V$-seminorms on $B(\cl H)$ with $E(W) \subseteq W$ always
 exists. This can be seen by invoking Zorn's lemma. In this case any
 chain $\{ E_{\lambda} \}$ will have a lower bound as can be seen by
 taking a weak*-limit point of the chain and noting that the limiting
 map $E$ will still satisfy, $E(W) \subseteq W.$
\end{remark}

\begin{prob} Let $V \subseteq W \subseteq B(\cl H)$ with $W$ admissable. Clearly, $I(V) \subseteq I^W(V)$ are they always equal? Is it possible to give necessary and sufficient conditions that guarantee equality?
\end{prob}

The following shows why we believe that the above problem is
important.

\begin{defn}
A C*-algebra $B$ is said to be QWEP( for quotient of WEP),
if there is a C*-algebra $A$ with WEP and a *-homomorphism from $A$
onto $B.$
\end{defn}

\begin{prob} Does a C*-algebra $B \subseteq B(\cl H)$ have QWEP if and
  only if $I^{B^{\prime \prime}}(B) = I(B)$ ?
\end{prob}

In the next section we shall relate these quantities to questions about weak expectations.


\section{Weak Expectations and Minimal Projections} 

We now turn our attention to some applications of the ideas of the
previous section to the existence of weak expectations.

\begin{defn}
Given $V \subseteq B(\cl H),$ an operator space,
we shall let $\cl E(V)= \cl E_{B(\cl H)}(V)$ denote the set of minimal
$V$-projections on $B(\cl H)$. Given $E \in \cl E(V),$ we call $\cl R(E)$ a {\bf copy of
$I(V)$} and we denote the set of all operator spaces contained in
$B(\cl H)$ that are copies of $I(V)$ by $\cl C \cl I(V).$
\end{defn}

Note that, in general, for each $\cl S \in \cl C \cl I(V),$ there
could be many projections, $E \in \cl E(V)$ with $S= \cl R(E).$  
We will often use the following observation of Hamana\cite{H1}, that
if $E_0, E_1 \in \cl E(\cl A),$ then $E_0 \circ E_1 \in \cl E(V)$
since it must also define a minimal $V$-seminorm.  Hence,
$\cl R(E_0 \circ E_1) = \cl R(E_0)$ and $E_0: \cl R(E_1) \to \cl
R(E_1)$ is a complete isomorphism.


Given a concrete operator space $V \subseteq B(\cl H),$ we shall let $V^{\dag \dag}$ denote the weak*-closure of $V$ in $B(\cl H),$ so that in the case that $V= \cl A$ is a C*-subalgebra, we have that $\cl A^{\dag \dag} = \cl A''.$

\begin{thm} Let $V \subseteq B(\cl H),$ be an operator space. Then the following are equivalent:
\begin{itemize} 
\item[(i)] $V$ has a weak expectation,
\item[(ii)] $I_{V^{\dag \dag}}(V) = I(V),$ i.e., these spaces are
  completely isometrically isomorphic via a map that fixes $V,$
\item[(iii)] $I_{V^{\dag \dag}}(V)$ is injective, in the usual sense.
\end{itemize}
\end{thm}
\begin{proof}

Next, assuming (ii), since $I(V)$ is injective we have (iii).

Assuming (iii), we have a projection onto, $I_{V^{\dag \dag}}(V) \subseteq V^{\dag \dag},$ and so $V$ has a weak
expectation. Thus, (iii) implies (i).

Finally, assuming (i), we have $E \in \cl E(V)$ with $\cl R(E)
\subseteq V^{\dag \dag}.$ Let $\phi: V^{\dag \dag} \to V^{\dag \dag}$ be any minimal
$V$-projection, relative to $V^{\dag \dag}$ so that $\phi(V^{\dag \dag})$ is a
copy of $I_{V^{\dag \dag}}(V).$ Since the seminorm on $B(\cl H)$
generated by $E$ is minimal among all $V$-seminorms, it is equal to the seminorm on $B(\cl H)$
generated by $\phi \circ E,$ and hence, $\phi \circ E \in \cl E(V).$
Also, $\phi: E(B(\cl H)) \to (\phi \circ E)(B(\cl H))$
is a complete isometry, since the two seminorms agree. But since $\phi$ is minimal among $V$-seminorms on $V^{\dag \dag}$, we have $(\phi \circ E)(V^{\dag \dag}) = \phi(V^{\dag \dag}),$ and so $(\phi \circ
E)(B(\cl H))= (\phi \circ E)(V^{\dag \dag})= \phi(V^{\dag \dag}).$ Thus, $\phi$ is
a complete isometry from a copy of $I(V),$ namely, $\cl R(E)$ onto a copy of
$I_{V^{\dag \dag}}(V).$
\end{proof}

In the C*-algebra case we can say a bit more.

\begin{prop} Let $\cl A \subseteq B(\cl H),$ be a unital
  $C^*$-subalgebra. If $\cl A$ has a weak expectation,
then $E(\cl A'') = \cl R(E)$ for every $E \in \cl E(\cl A).$
\end{prop}

\begin{proof} If $\cl A$ has a weak expectation, then there exists a copy of
  $I(\cl A)$ inside $\cl A''.$ Consequently, there exists $E_0 \in \cl
  E(\cl A)$ such that $\cl R(E_0) \subseteq \cl A''.$ Since $E_0$ is a
  projection, $\cl R(E_0)= E_0(\cl A'').$ Now given any, $E \in \cl
  E(\cl A),$ we have that $\cl R(E) = \cl R(E \circ E_0) =E(\cl
  R(E_0)) = E(E_0(\cl A'')) \subseteq E(\cl A'')$ and so, $E(\cl A'')
  = \cl R(E).$ 
 \end{proof}
 
 \begin{prob} Let $\cl A \subseteq B(\cl H),$ be a unital
$C^*$-subalgebra. Does there exist $E \in \cl E(\cl A),$ such that
$E(\cl A'')= \cl R(E) \cap \cl A'' $ is a copy of $I_{\cl A''}(\cl A)$?
\end{prob}

If the answer to the above problem is affirmative, then the converse of the above proposition holds.

Let 
$\cl K(\cl H)$ denote the
ideal of compact operators on $\cl H$. We now turn our attention to
the relationship between compact operators and weak expectations. Hamana\cite{H1} proves that if $\cl A \subseteq B(\cl H)$ is a unital C*-saubalgebra 
and $\cl K(\cl H) \subset \cl A$,
then  $I(\cl A) = B(\cl H)$.
The following is a slight generalization and in the C*-algebra case yields a different proof.  This proof also serves to introduce 
some of the ideas of the next section.

\begin{prop}
Let $\cl K(\cl H) \subseteq V \subseteq B(\cl H)$ be an operator space.
Then the identity map on $B(\cl H)$ is the only $V$-map and 
consequently,
$I(V) = B(\cl H).$
\end{prop}
\begin{proof}
Let $\phi$ be an $\cl A$-map. We first show that $\phi(I) = I,$ so that $\phi$ is a unital complete contraction and hence completely positive.

To see this claim , note that for any finite rank projection $P$, we have that $\| (P, \phi(I) -P)\| = \|( \phi(P), \phi(I -P))\| \le \|(P, I-P)\| = 1.$ Hence,
$(\phi(I) - P)(\phi(I) -P)^* \le I -P.$ Multiplying both sides by $P$ yields that $0= P(\phi(I) -P) = P(\phi(I) -I).$ Since this holds for every finite rank projection, $\phi(I) =I.$
 
Every positive operator $R \in B(\cl H)$ is the strong limit of 
an increasing net(sequence
in the separable case) of finite rank positive operators, $\{ 
F_{\alpha} \}$.
Thus, we have that $F_{\alpha} = \phi(F_{\alpha}) \le \phi(R)$.  
Taking limits, we have that
$R \le \phi(R)$. Choosing a scalar, $r$ such that, $0 \le rI - 
R$ we have that
$rI - R \le \phi(rI - R) = rI - \phi(R)$ and hence $R=\phi(R).$  Since every 
operator is a sum of positive operators
the result follows.
\end{proof}

In the case of C*-subalgebras, we can someting about the opposite extreme, $\cl K(\cl H) \cap \cl A = 
0$.

\begin{prop}
Let $\cl A \subseteq B(\cl H)$ be a unital $C^*$-subalgebra and 
assume that $\cl K(\cl H) \cap \cl A = 0$.
Then for every copy $\cl S$ of $I(\cl A)$, there 
exists a minimal
$\cl A$-projection $E$ onto $\cl S$ with $E(\cl K(\cl H)) =0$ and hence,
$\cl S \cap \cl K(\cl H) = (0).$ 
However, there can exist minimal $\cl A$-projections
with $E(\cl K(\cl H)) \ne 0.$
\end{prop}
\begin{proof}
Consider the projection map $\pi$ of $B(\cl H)$ onto the Calkin 
algebra, $\cl Q(\cl H).$
Since this map is a *-isomorphism on $\cl A,$ by rigidity, it 
must be a complete isometry
on every copy of $I(\cl A)$. Thus, by composition with $\pi$ one 
is able to obtain an $\cl A$-projection
that vanishes on the compacts.

For an example of a minimal $\cl A$-projection onto a copy of 
$I(\cl A)$ that does not vanish on the
compacts, consider the case when $\cl A$ consists of the scalar 
multiples of the identity operator.
Fix a unit vector $h \in \cl H$ and set $E(T) = \langle Th,h 
\rangle I$.
\end{proof}

\begin{prob} Does it also follow that every copy of $I_{\cl
    A''}(\cl A)$ and $I^{\cl A''}(\cl A)$
  intersects the compacts trivially ?
\end{prob}

In the case that $\cl A$ is irreducible, the above result can be improved.

\begin{prop}
Let $\cl A \subseteq B(\cl H)$ be a unital, irreducible 
$C^*$-subalgebra and assume that 
$\cl K(\cl H) \cap \cl A = 0$.
Then $\phi(\cl K(\cl H)) \subseteq \cl K(\cl H)$ for every $\cl 
A$-map.
\end{prop}
\begin{proof}
Let $P$ and $Q$ be fixed finite rank projections with $PQ=0.$ By 
the extension of Kaplansky's
density theorem\cite{KR}, the exists a unitary $U \in \cl A$ 
such that $U= P-Q$ on the range of
$P+Q$. Since $\|U\|=\|P-Q\| =1$, the range of $P+Q$ reduces $U$.

Let $H= Re(U)$ and write $H = H^+ - H^-$. Then $H^+ \in \cl A$, 
$H^+ \ge P$ and on the range of $P+Q$, $H^+ = P$. Constructing 
one such element of $\cl A$ for
each $Q$ and letting $Q$ tend strongly to $I-P$, we obtain a net 
of elements $H_Q \in \cl A$,
such that $H_Q \ge P$ and converges strongly to $P.$ Since $H_Q 
= \phi(H_Q) \ge \phi(P)$,
we find that $P \ge \phi(P) \ge 0.$

Thus $\phi$ is rank reducing, and the result follows.
\end{proof}

\begin{prop}
Let $\cl A \subseteq B(\cl H)$ be a unital, irreducible 
$C^*$-subalgebra and assume that 
$\cl K(\cl H) \cap \cl A = 0$.
If $E$ is a minimal $\cl A$-projection, then
$E(\cl K(\cl H)) =0$.
\end{prop}
\begin{proof}
Let $\cl S$ be the range of $E$ and let $\gamma : \cl Q(\cl H) 
\to \cl S$ be a completely positive
map with $\gamma(\pi(a)) = a $ for every $a \in \cl A$. Since 
$p_E$ is a minimal $\cl A$-seminorm we
have that $\gamma \circ \pi \circ E$ defines the same seminorm.

However, by the above result, this latter seminorm vanishes on 
the compact operators and hence $E$
must also vanish on the compact operators.
\end{proof}

For the next result, we need to recall the canonical decomposition of
a completely bounded map, $\phi: B(\cl H) \to B(\cl H)$ into a
singular and absolutely continuous part, $\phi=
\phi_s + \phi_{ac}.$ This decomposition is achieved by considering the
generalized Stinespring representation, $\phi(x) = V^*\pi(x)W,$ where
$\pi: B(\cl H) \to B(\cl H_0)$ is a *-homomorphism and $V,W: \cl H \to
\cl H_0$ are bounded linear maps and decomposing the *-homomorphism into
it's singular and absolutely continuous parts, $\pi= \pi_s \oplus \pi_{ac}$. 
Recall that this latter decomposition is gotten by setting, $\cl
H_{ac} = \pi(\cl K(\cl H))\cl H_0,$ and $\cl H_s = \cl H_0 \ominus \cl
H_{ac}.$ When $\phi$ is completely positive, it is easy to see that
$\phi_s$ and $\phi_{ac}$ are also completely positive.  We will also use the fact that $\phi_{ac}$ is weak*-continuous.

We also remind the reader that a self-adjoint algebra of operators,
$\cl B$ on
a Hilbert space, $\cl H$, is said to {\em act non-degenerately} if the
closed linear span, $\cl B
\cl H$ is dense in $\cl H.$

\begin{prop} Let $\cl A \subseteq B(\cl H)$ be a unital $C^*$-subalgebra and
  let $\phi$ be an $\cl A$-map. If $\cl A \cap \cl K(\cl H)$ acts
  non-degenerately, then $\phi= \phi_{ac}$ and consequently, $\phi$ is
  an $\cl A''$-map.
\end{prop}
\begin{proof} Since $\cl A \cap \cl K(\cl H)$ is non-degenerate, we
  can find an increasing net of compact operators, $K_i \in \cl A$ which tend strongly to
  the identity. Since, $\phi_s(K_i)=0, K_i= \phi_{ac}(K_i) \le \phi_{ac}(I).$
  But since these operators tend to the identity, $\phi_{ac}(I)=I,$ and
  hence $\phi_s(I)=0,$ which implies, $\phi_s=0,$ because it is completely
  positive.

Finally, since $\phi=\phi_{ac}$ is absolutely continuous and fixes $\cl A$,
it must also fix $\cl A''.$
\end{proof}

\begin{lemma} Let $\cl A \subseteq B(\cl H)$ be a unital $C^*$-subalgebra. If
  $\cl A \cap \cl K(\cl H)$ acts non-degenerately, then $\cl A''$ is injective.
\end{lemma}
\begin{proof} Let $\cl B= \cl A \cap \cl K(\cl H).$ By
  \cite[Theorem~1.4.5]{Ar} and its proof $\cl B$ is unitarily equivalent to the direct sum
  of elementary $C^*$-algebras where each algebra appears with certain
  multiplicities. Thus, after this unitary equivalence we have that
  $\cl H = \sum_i n_i \cl H_i,$ where $n_i$ indicates the multiplicity
  with which $\cl H_i$ occurs and
$\cl B = \{ \sum_i \oplus n_i K_i: K_i \in \cl K(\cl H_i) \}.$

Since $\cl B = \cl A \cap \cl K(\cl H),$ we have that after the
unitray equivalence, each $A \in \cl
A$ is necessarily of the form $A = \sum_i \oplus n_i A_i,$ with $A_i
\in B(\cl H_i).$
Hence, $\cl B'' \subseteq  \cl A'' \subseteq \{ \sum_i n_i B_i: B_i
\in B(\cl H_i), \sup_i \|B_i\| < \infty \} = \cl B'',$ and it follows that
$$\cl A'' = \cl B'' = \{ \sum_i \oplus n_i B_i: B_i \in B(\cl H_i),
\sup_i \|B_i\| < \infty \},$$
which is clearly injective.
\end{proof}

\begin{thm} Let $\cl A \subseteq B(\cl H)$ be a unital
  C*-subalgebra. If $\cl A \cap \cl K(\cl H)$ acts non-degenerately
  and $E \in \cl E(\cl A),$ then $E$ is weak*-continuous and $\cl R(E) = \cl A''.$ Consequently,
  $I(\cl A)= \cl A''$ and $\cl A''$ is the unique copy of $I(\cl A)$
  contained in $B(\cl H).$
\end{thm}
\begin{proof} By Proposition~3.8, $E=E_{ac}$ so that $E$ is
  weak*-continuous and also $E$ fixes $\cl A''$ so that $\cl A''
  \subseteq \cl R(E).$ But since $\cl A''$ is injective, $\cl R(E)
  \subseteq \cl A''$ for any minimal $\cl A$-projection and the result
  follows.
\end{proof}

Earlier, we saw that if we fix $E \in \cl E(\cl A)$, then the set $\cl
J_E$ is a convex left ideal in the semigroup of all $\cl
A$-maps. Note that in this case $\cl J_E$ is also closed 
in the point-weak*-topology. We also have that $\cl J_E$ is left
invariant under the action of the smaller convex semigroup $\Gamma$ of normal $\cl 
A$-maps. Consequently, there exist minimal $\Gamma$-invariant 
subsets of $\cl E(\cl A).$

\begin{prop}
Let $\cl A \subseteq B(\cl H)$ be a unital $C^*$-subalgebra. Then $\cl 
A$ has a weak expectation if and only if there exists a minimal 
$\Gamma$-invariant subset of $\cl E(\cl A)$ that is a singleton.
\end{prop}
\begin{proof}
Let $E$ be $\Gamma$-invariant subset that is a singleton and let 
$U \in \cl A'$ be a 
unitary, then $U^*E(x)U = E(x)$ for all $x$ and it follows that 
the range of $E$ is contained in $\cl A''$.

Conversely,
by a theorem of Haagerup every element of $\Gamma$ has the form
$\phi(x) = \sum b^*_ixb_i$ for some sequence of elements in 
$\cl A'$ satisfying $\sum b^*_ib_i = 1.$
If $\cl A$ has a weak expectation, then there is a minimal $\cl 
A$-projection, $E$, whose
range is contained in $\cl A''$, and that element is fixed by 
$\Gamma.$
\end{proof}

If $\cl A$ has a weak expectation, does every minimal 
$\Gamma$-invariant subset have to be a singleton?


\section{The Fixed Space}

In this section we study the set of elements that are fixed by 
all $V$-maps.  We first show that this is a type of reflexive cover of $V$ and then in the case of a C*-subalgebra, we show that this set can be identified
with a type of order completion.

\begin{defn}
Let $V \subseteq W$ be operator spaces. We set
$\cl F_{W}(V) = \{ T \in W : \phi(T) = T \text{ for every
V-map, }  \phi: W \to W \}.$ When $V \subseteq B(\cl H)$ we shall write $\cl F(V) \equiv \cl F_{B(\cl H)}(V).$
\end{defn}

We shall generally be concerned with the case where $V= \cl A$ and $W= \cl
B$ are C*-algebras, but the case of a pair of operator spaces is
equally interesting. 
The following is immediate.

\begin{prop} Let $V \subseteq W \subseteq \cl B(\cl H),$ then $\cl F(V) \cap W \subseteq \cl F_{W}(V).$
\end{prop}
 
 In general we won't have equality. To see this note that if $\cl A \subseteq \cl S \subseteq B(\cl H),$ where $\cl S$ is one of the copies of $I(\cl A)$, then $\cl F_{\cl S}(\cl A) = \cl S,$ since every completely positive map from $\cl S$ to $\cl S$ that fixes $\cl A$ necessarily fixes all of $\cl S,$ by the rigidity property of injective envelopes. But there can be other copies of the injective envelope embedded in $B(\cl H)$ and by taking a projection onto one of these other copies, we obtain a completely positive map on $B(\cl H)$ that fixes $\cl A$ but doesn't fix $\cl S.$

The following result shows that $\cl F(V)$ is a sort of reflexive cover of $V.$

\begin{prop}
Let $V \subseteq B(\cl H)$ be an operator space. Then
$\cl F(V) = \{ T \in B(\cl H) : E(T) = T \text{ for every } E \in \cl 
E(V) \}
= \bigcap  \cl S$, where the intersection is taken over all  $\cl S \in \cl C \cl I(V)$, i.e., over all copies of $I(V)$. Consequently, $I(V) = I(\cl F(V)).$
\end{prop}

\begin{proof}  The equality of the last two sets is obvious, as is the fact that the first set is contained in the second set. Now if $E(T)=T$ for every $E \in \cl E(V)$ and $\phi$ is any $V$-map, then $\phi(T) = \phi(E(T)) = T,$ since $\phi \circ E \in \cl E(V).$
\end{proof}

In the case of a C*-subalgebra we can say a bit more.

\begin{prop}
Let $\cl A \subseteq B(\cl H)$ be a unital $C^*$-subalgebra. Then $\cl 
F(\cl A) \subseteq \cl F_{\cl A''}(\cl A) \subseteq \cl A''.$ Moreover, $I_{\cl A''}(\cl A) = I_{\cl A''}(\cl F(\cl A)).$
\end{prop}
\begin{proof}
Let $P \in \cl A'$ be a projection and define an $\cl A$-map by 
$\phi(X) = PXP + (I-P)X(I-P)$.
If $T \in \cl F(\cl A)$, then $\phi(T) = T$ and hence $TP=PT$.  
Since $T$ commutes with
every projection in $\cl A'$, we have that $T \in \cl A''.$
Thus, $\cl F(\cl A) \subseteq \cl A''.$ Applying the Proposition~4.2, we have that $\cl F(\cl A) = \cl F(\cl A) \cap \cl A'' \subseteq \cl F_{\cl A''}(\cl A).$

The last statement follows since any unital completely positive map that fixes $\cl A$ fixes $\cl F(\cl A).$
\end{proof}

\begin{prob} Is $\cl F(V) \subseteq V^{\dag \dag},$ for every operator space?
\end{prob}

\begin{remark} Note that by the above results, if $\cl A= \cl A'' \subseteq B(\cl H)$ is a non-injective vonNeumann subalgebra, then $\cl A = \cl F(\cl A) = \bigcap \cl S$ and so there are certainly multiple copies of $I(\cl A)$. Later we will see an example of a non-injective C*-algebra for which there is a unique copy of $I(\cl A)$ and $I(\cl A) \subsetneq \cl A''.$
\end{remark}

\begin{prop} Let $\cl A \subseteq B(\cl H)$ be a unital $C^*$-subalgebra. If $\cl A \cap \cl K(\cl H)$ acts non-degenerately, then $\cl F(\cl A) = \cl A''.$
\end{prop}
\begin{proof}  By Theorem~3.10, $\cl A'' \subseteq \cl F(\cl A).$
\end{proof}

\begin{prop}  Let $\cl A \subseteq B(\cl H)$ be a unital $C^*$-subalgebra. If $\cl A \cap \cl K(\cl H) = (0),$ then $\cl F(\cl A) \cap \cl K(\cl H) = (0).$
\end{prop}
\begin{proof} Apply Proposition~3.4.
\end{proof}

The next result shows that many of the various notions of multipliers that can be associated with $\cl A$ are in $\cl F(\cl A).$

\begin{defn} Let $\cl A \subseteq B(\cl H)$ be a unital
  $C^*$-subalgebra. An operator $T \in B(\cl H)$ is called a {\bf
    local left multiplier of $\cl A,$} provided that there exists a
  two sided ideal, $\cl J \lhd \cl A$ of $\cl A$ that acts non-degenerately on $\cl H,$ such that $T \cdot \cl J \subseteq \cl A.$
Similarly, $T$ is called a {\bf local right multiplier of $\cl A$} if there exists such an ideal with $\cl J \cdot T \subseteq \cl A,$ and a {\bf local quasi-multiplier of $\cl A,$} provided that there is such an ideal with $\cl J \cdot T \cdot \cl J \subseteq \cl A.$  These sets of operators are denoted $\cl L \cl M_{loc}(\cl A), \cl R \cl M_{loc}(\cl A)$ and $\cl Q \cl M_{loc}(\cl A),$ respectively.
\end{defn}

It is fairly easily checked that each of these sets of operators is a vector space and that $\cl L \cl M_{loc}(\cl A) \cup \cl R \cl M_{loc}(\cl A) \subseteq \cl Q \cl M_{loc}(\cl A).$ Moreover, $T \in \cl Q \cl M_{loc}(\cl A)$ if and only if $Re(T), Im(T) \in \cl Q \cl M_{loc}(\cl A).$

\begin{prop} Let $\cl A \subseteq B(\cl H)$ be a unital
  $C^*$-subalgebra. \\
Then $\cl Q \cl M_{loc}(\cl A) \subseteq \cl F(\cl A).$
\end{prop}
\begin{proof} Let $T \in \cl Q \cl M_{loc}(\cl A)$ and let $\cl J \lhd \cl A$ be a two-sided ideal that acts non-degenerately on $\cl H,$ with $\cl J \cdot T \cdot \cl J \subseteq \cl A.$ We wish to show that $T \in \cl F(\cl A). $ 

Since $\cl J$ acts non-degenerately, there is an increasing, positive, approximate identity $\{e_{\alpha} \}$ for $\cl J$ that tends strongly to $I.$  Let $\phi$ be any $\cl A$-map, so that $\phi$ is a unital completely positive map that fixes $\cl A.$  

By Choi's theory of multiplicative domains \cite{Ch}, $\phi$ is an $\cl A$-bimodule map. 
Hence, for each $\alpha$ and $\beta,$ $(e_{\alpha}T - e_{\alpha}\phi(T))e_{\beta} = 0.$
Using that $e_{\beta}$ tends strongly to $I$ yields that $e_{\alpha}T
= e_{\alpha} \phi(T),$ for each $\alpha.$ Now using that $e_{\alpha}$
tends strongly to $I,$ yields $T = \phi(T).$ Thus, $T \in \cl F(\cl A)$ as was to be shown.
\end{proof}

We now show that $\cl F(\cl A)$ is in a certain sense an order completion of $\cl A.$

\begin{defn}
Let $T = T^* \in B(\cl H)$, then we set $(- \infty, T]_{\cl A} = 
\{ A=A^* \in \cl A : A \le T \}$
and set $[T, + \infty)_{\cl A} = \{ A=A^* \in \cl A : T \le A 
\}$.
For $R= R^*$, we write $(-\infty, T]_{\cl A} \le R$ provided 
that $A \le R$ for every 
$A \in (- \infty, T]_{\cl A}$ and define $R \le [T, 
+\infty)_{\cl A}$, similarly.
We say that $T$ is {\bf order determined by $\cl A$}, if 
$(-\infty, T]_{\cl A} \le R \le [T, +\infty)_{\cl A}$,
implies that $T=R$.

We say that $T$ is {\bf matricially order determined by $\cl A$} 
provided that
$(-\infty, T \otimes H]_{M_n(\cl A)} \le R \otimes H \le [T 
\otimes H, +\infty)_{M_n(\cl A)}$
for every $n$ and every $H=H^* \in M_n$, implies that $R=T$.
\end{defn}

Note that the set of order determined elements is a subset of 
the matricially order determined elements.

\begin{prop}
Let $\cl A \subseteq B(\cl H)$ be a unital $C^*$-subalgebra and let $T=T^* \in B(\cl H),$ then $T \in \cl F(\cl A)$ if and only 
if $T$ is matricially 
order determined by $\cl A.$
\end{prop}
\begin{proof}
Let $\phi$ be any $\cl A$-map. It is easily seen that
$(-\infty, T \otimes H]_{M_n(\cl A)} \le \phi(T) \otimes H \le 
[T \otimes H, +\infty)_{M_n(\cl A)}$
for every $n$ and every $H=H^* \in M_n$. Hence, if $T$ is 
matricially order determined, then $\phi(T) = T$
and so $T \in \cl F(\cl A).$

Conversely, assume that $T \in \cl F(\cl A)$ and assume that 
$R=R^*$ satisfies the inequalities in the
definition. These inequalities imply that there is a 
well-defined completely positive map, $\phi$,
satisfying, $\phi(A + \lambda T) = A + \lambda R$, from the 
operator system spanned by $\cl A$ and
$T$ onto the operator system spanned by $\cl A$ and $R$. This 
completely positive map can then
be extended to a completely positive map on all of $B(\cl H)$, 
and hence $R= \phi (T) = T.$
\end{proof}

Maitland Wright\cite{MW1, MW2, MW3} and
Hamana\cite{H4} studied several different monotone completions of a C*-algebra.  In spite of the above characterization of $\cl F(\cl A),$ we have been unable to develop any relationship between $\cl F(\cl A)$ and those other completions.

\begin{remark} The set $\cl F(\cl A)$ is not generally a C*-subalgebra of $B(\cl H)$ as we will show below.
\end{remark}

We now wish to recall another construction of Hamana's\cite{H4}. Given a concrete operator space, $V \subset B(\cl H)$, we have a new operator space $\widehat{V} \subseteq B(\cl H \otimes \ell^2)$ defined by as follows.  Every $T \in B(\cl H \otimes \ell^2)$ has the form $T=(T_{i,j}),$ with $T_{i,j} \in B(\cl H).$  We set $\widehat{V}= \{ (T_{i,j}) \in B(\cl H \otimes \ell^2) : T_{i,j} \in V \}.$
This is Hamana's  {\em Fubini product} of $V$ with $B(\ell^2).$

If $\phi: B(\cl H) \to B(\cl H)$ is completely bounded, then $\widehat{\phi}: B(\cl H \otimes \ell^2) \to B(\cl H \otimes \ell^2)$, defined by $\widehat{\phi}((T_{i,j}))= (\phi(T_{i,j}))$ is also completely bounded with $\|\phi\|_{cb}= \| \widehat{\phi}\|_{cb}.$ Hence, if $V$ is injective and $\phi: B(\cl H) \to V$ is a completely contractive projection onto $V$, then $\widehat{\phi}: B(\cl H \otimes \ell^2) \to \widehat{V}$ is a completely contractive projection onto $\widehat{V}$ and so $\widehat{V}$ is also injective.

We now wish to define another operator space associated with
$V$. First, let $\ell^{\infty}(V)$ denote the subset of $\widehat{V}$
consisting of diagonal matrices with entries from $V$ and let $V
\otimes \cl K(\ell^2) \subseteq \widehat{V}$ denote the tensor product
of $V$ and the compact opertors on $\ell^2, \cl K(\ell^2).$ Finally,
we let $ \widetilde{V}= \ell^{\infty}(V) + V \otimes \cl K(\ell^2).$

\begin{prop} Let $\cl A \subseteq B(\cl H)$ be a unital, C*-subalgebra. Then $\widetilde{\cl A} \subseteq B(\cl H \otimes \ell^2)$ is a C*-subalgebra. If $\Psi: B(\cl H \otimes \ell^2) \to B(\cl H \otimes \ell^2)$ is a unital, completely positive map that fixes $\widetilde{\cl A}$,
then there exists, $\phi: B(\cl H) \to B(\cl H)$ that fixes $\cl A$ such that $\Psi= \widehat{\phi}.$
\end{prop}
\begin{proof} Since $\Psi$ fixes the C*-algebra, $\cl C = \ell^{\infty}(\bb C
  \cdot I) + ( \bb C \cdot I \otimes \cl K(\ell^2))$, we have that
  $\Psi$ must be a $\cl C$-bimodule map and hence, must be of the form, $\Psi = \widehat{\phi}$ for some map $\phi$.
The result now follows easily.
\end{proof}

\begin{thm} Let $\cl A \subseteq B(\cl H)$ be a unital C*-subalgebra. Then $\Psi \in \cl E(\widetilde{\cl A})$ if and only if $\Psi= \widehat{\phi}$ for some $\phi \in \cl E(\cl A)$ and $\cl F(\widetilde{\cl A}) = \widehat{\cl F(\cl A)}.$
\end{thm}
\begin{proof} The fact that $\Psi= \widehat{\phi}$ for $\phi \in \cl E(\cl A)$ follows from the above Proposition. Note that in this case, $\cl R(\Psi)= \widehat{\cl R(\phi)},$ and the result follows from the characterization of $\cl F( \cdot)$ as the intersection of all ranges.
\end{proof}

\begin{cor} Let $\cl A \subseteq B(\cl H)$ be a unital injective C*-subalgebra.
Then $\widetilde{\cl A} \subseteq B(\cl H \otimes \ell^2)$ is not injective, and $\cl F(\widetilde{\cl A}) =I(\widetilde{\cl A})= \widehat{\cl A} \subseteq B(\cl H \otimes \ell^2)$ is the unique copy of its injective envelope.
\end{cor}

We now wish to show that it is possible to find an abelian, injective C*-subalgebra, $\cl A \subseteq B(\cl H)$ such that $\widehat{\cl A} \subseteq B(\cl H \otimes \ell^2)$ is not a C*-subalgebra.
First, we will need some preliminary results which might be of independent interest.

Let $X$ be a compact, Hausdorff space. If we identify $I(C(X))=C(Y)$ for some compact, Hausdorff space $Y,$ then the inclusion of $C(X)$ in $C(Y)$ is induced by a continuous, onto function, $p:Y \to X.$ Assume that $\{x_n \}$ is a countable, dense subset of $X$ and choose, $y_n \in Y$ such that $p(y_n)=x_n.$
In the following sequence of results we assume that this situation holds.

\begin{prop} Define $\pi: C(Y) \to \ell^{\infty}$ by $\pi(f)= (f(y_n)).$ Then $\pi$ is a one-to-one *-homomorphism and consequently, $\{ y_n \}$ is dense in $Y.$
\end{prop}
\begin{proof} Clearly, $\pi$ is a *-homomorphism. Define $\rho: C(X) \to \ell^{\infty}$ by $\rho(f)= (f(x_n)).$ Since $\{ x_n \}$ is dense in $X, \rho$ is a one-to-one *-homomorphism. Also, since $p(y_n)=x_n, \pi$ is an extension of $\rho$ to $C(Y),$ i.e., $\rho= \pi \circ p^*.$

But since $C(Y)$ is an essential extension of $C(X)$, the fact that $\rho$ is isometric forces $\pi$ to be isometric and hence $\{ y_n \}$ must be a dense subset of $Y$. 
\end{proof}

The above proof gives one of the easiest proofs of the following result.

\begin{cor} Let $X$ and $Y$ be compact, Hausdorff spaces, such that $C(Y)$ is $C^*$-isomorphic to $I(C(X)).$  If $X$ is separable, then $Y$ is separable.
\end{cor}

Now let $X= [0,1]$ and let $\{x_n \}$ be a dense subset as above of distinct points and for convenience we let $x_1=1.$ Note that $\rho(C([0,1])) \cap c_0 = (0),$ for if $\rho(f) \in c_0,$ then given any $x \in [0,1]$ we could choose a subsequence, $\{x_{n_k} \}$ with $\lim_k x_{n_k} = x$ and hence $f(x) = \lim_k f(x_{n_k}) =0.$

Thus, if we let $\tilde{\rho}: C([0,1]) \to \ell^{\infty}/c_0$ denote the composition of $\rho$ with the quotient map, then $\tilde{\rho}$ is still one-to-one and hence an isometry. Thus, again by the fact that $C(Y)$ is an essential extension of $C([0,1])$, we have that the composition of $\pi$ with the quotient map, $\tilde{\pi}: C(Y) \to \ell^{\infty}/c_0$ is also an isometry and hence one-to-one.

\begin{lemma}\label{B} There exists an injective, C*-subalgebra, $\cl B$ of $\ell^{\infty}$ such that $\cl B \cap c_0 = (0), \cl B$ is wk*-dense in $\ell^{\infty}$ and for every $n$, there exists a  strictly positive element, $b_n \in \cl B$ such that $\lim _m b_n^m = \delta_n,$ pointwise, where $\delta_n$ is the function that is 1 at $n$ and 0 elsewhere.
\end{lemma}
\begin{proof} Let $\cl B = \pi(C(Y))$ where $C(Y)$ is the injective envelope of $C([0,1])$ as above. Then $\cl B$ is injective and as shown above, $\cl B \cap c_0 = (0).$

Next, for each point, $x_n \in [0,1]$ choose a strictly positive continuous function, $f_n$ such that, $1=f_n(x_n) >f_n(x) \ge 1/2$ for any $x \ne x_n$ and let $b_n= \rho(f_n).$

The existence of such $b_n$ shows that each $\delta_n$ is in the wk*-closure of $\cl B$ and hence, $\cl B$ is wk*-dense.
\end{proof}

The key to the next result is that products in $\widehat{\cl B}$ involve strong convergence of sums while injective C*-subalgebras need not be closed in the strong operator topology.

\begin{thm} Let $\cl B \subseteq \ell^{\infty} \subseteq B(\ell^2)$ be represented as diagonal operators be the C*-algebra of Lemma~\ref{B}.
Then $\cl F(\widetilde{\cl B})= I(\widetilde{\cl B}) = \widehat{\cl B}$ is not a C*-subalgebra of $B(\cl H \otimes \ell^2).$
\end{thm}
\begin{proof} Since $\cl B$ is injective, $\cl F(\cl B)= \cl B,$ and hence, invoking the above Proposition, $\cl F(\widetilde{\cl B})= \widehat{\cl F(\cl B)}= \widehat{\cl B}.$  Thus, it remains to show that even though $\cl B$ is an injective C*-subalgebra of $B(\ell^2), \widehat{\cl B}$ is not a C*-subalgebra of $B(\ell^2 \otimes \ell^2),$ although by the above results it is an injective operator system.

To this end, let $b_1$ be the element of $\cl B$ that satisfies,
$1=b_1(1) > b_1(n) \ge 1/2,$ for $n \ne 1.$ Set $P_1= \sqrt{b_1}$ and
for $n>1,$ set $P_n= \sqrt{b_1^{n-1}-b_1^n}.$ If we let $A=(A_{i,j}),
B=(B_{i,j})$ be the defined by $A_{1,j}= P_j, A_{i,j}=0, i \ne 1$ and
$B_{1,1}= P_1, B_{i,1}= -P_i, B_{i,j}=0, j \ne 1,$ then since, $\sum_k
P_k^2 \le 2b_1$, we have that $A$ and $B$ define bounded operators and hence are in $\widehat{\cl B}.$

However, $A \cdot B= (C_{i,j})$ where $C_{i,j}= 0,$ unless $i=j=1$ and $C_{1,1}= P_1^2 - \sum_{k=2} P_k^2= \lim_{k \to \infty} b_1^k = \delta_1$ since all convergence is only in the strong operator topology. However, $\delta_1 \notin \cl B$ by construction. Hence, $\widehat{\cl B}$ is not a C*-subalgebra.
\end{proof}

Note that for the above example, $\cl F(\widetilde{\cl B})=
\widetilde{\cl B} =I(\widetilde{\cl B}).$ Thus, although $\cl
F(\widetilde{\cl B})$ is not a C*-subalgebra of $B(\cl H \otimes
\ell^2)$, it is an injective operator system it is a C*-algebra in another product. 

\begin{prob} Is $\cl F(\cl A)$ always completely order isomorphic to a C*-algebra?
\end{prob}

In particular, there is a natural way to identify $\cl F(\cl A)$ completely order isomorphically with an operator subsystem of $I(\cl A)$ and we conjecture that it is a C*-subalgebra of $I(\cl A)$ with this identification.
We make this precise below.

\begin{defn} Let $\cl A$ be a unital C*-algebra and let $\pi: \cl A
  \to B(\cl H)$ be a *-monomorphism. We let $\cl E(\pi)$ denote the set of
  all completely positive maps, $\phi: I(\cl A) \to B(\cl H)$ that
  extend $\pi$ and we set $\cl F_{\pi}= \{  x \in I(\cl A): \phi(x)=
  \psi(x) \text{ for all } \phi, \psi \in \cl E(\pi) \}.$
\end{defn}

\begin{prop} Let $\cl A$ be a unital C*-algebra and let $\pi: \cl A
  \to B(\cl H)$ be a *-monomorphism. If $\phi \in \cl E(\pi)$, then $\phi: \cl F_{\pi} \to \cl F(\pi(\cl A))$ is a complete order isomorphism.
\end{prop}
\begin{proof} Clearly, $E(T)=T$ for every $E \in \cl E(\pi(\cl A))$ if and only if $T=\phi(x)$ for a unique element of $\cl F_{\pi}.$
\end{proof}

\begin{prob} Is $\cl F_{\pi} \subseteq I(\cl A),$ always a C*-subalgebra?
\end{prob}

We close this section by examining the above construction in the abelian case.
Recall that an abelian C*-algebra is injective if and only if it is an AW*-algebra \cite{H1} and that, by definition, in an AW*-algebra every set of self-adjoint elements with an upper bound(respectively, lower bound) has a supremum(respectively, infimum). Also, every abelian W*-algebra is injective.
Given any subset $S$ of a C*-algebra, we let $S_h= \{ x \in S: x=x^* \}.$
 
Given an abelian, injective C*-algebra $\cl C$ and a C*-subalgebra
$\cl A$, we follow Hamana's notation \cite{H4} and given $x \in \cl C_h$ set
$$(-\infty, x]_{\cl A} = \{ a \in \cl A_h : a \le x \}$$
and
$$[x, +\infty)_{\cl A} = \{ a \in \cl A_h: x \le a \}.$$
Moreover, we shall set $$\ell_{\cl A}(x) = \sup (- \infty, x]_{\cl A}$$ and
$$u_{\cl A}(x) = \inf [x, + \infty)_{\cl A}.$$

\begin{thm} Let $\cl A \subset \cl B(\cl H)$ be a unital abelian C*-subalgebra with $\cl A^{\prime} = \cl A^{\prime \prime}.$  Then $\cl F(\cl A)_h = \cl F_{\cl A^{\prime \prime}}(\cl A)_h = \{ x \in \cl A^{\prime \prime}_h: \ell_{\cl A}(x) = u_{\cl A}(x) \}$, where $\ell_{\cl A}(x)$ and $u_{\cl A}(x)$ are computed in $\cl A^{\prime \prime}.$
\end{thm}
\begin{proof}
Let $E: B(\cl H) \to \cl A^{\prime \prime}$ be a completely positive projection. Since $E$ fixes $\cl A$, if $x \in \cl F(\cl A),$ then $E(x)=x$ and so $x \in \cl A^{\prime \prime}.$ Thus, $\cl F(\cl A) = \cl F(\cl A) \cap \cl A^{\prime \prime} \subseteq \cl F_{\cl A^{\prime \prime}}(\cl A).$

Now, if $x \in \cl A^{\prime \prime}_h$ and $\ell_{\cl A}(x) \ne u_{\cl A}(x),$ then there exists, $y \in \cl A^{\prime \prime}_h$ such that $\ell_{\cl A}(x) \le y \le u_{\cl A}(x)$ with $y \ne x.$

It is easily checked that for any $a \in \cl A$ and $\lambda \in \bb C,$ we have that $a + \lambda x \ge 0$ implies that $a+ \lambda y \ge 0.$ Hence, the map $\phi(a+ \lambda x) = a + \lambda y$ is positive and, since we are in an abelian situation, completely positive. Thus, we may extend $\phi$ to a map, $\psi: \cl A^{\prime \prime} \to \cl A^{\prime \prime}.$  Since $\psi$ fixes $\cl A$ and $\psi(x) =y,$ we have that $x \notin \cl F_{\cl A^{\prime \prime}}(\cl A).$  Thus, we have shown that $\cl F(\cl A)_h \subseteq \cl F_{\cl A^{\prime \prime}}(\cl A) \subseteq \{ x \in \cl A^{\prime \prime}_h: \ell_{\cl A}(x) = u_{\cl A}(x) \}.$

Now assume that $x \in \cl A^{\prime \prime}_h$ and $\ell_{\cl A}(x) = u_{\cl A}(x).$ 
If $\phi:B(\cl H) \to B(\cl H)$ is any completely positive map such that $\phi(a)=a \forall a \in \cl A,$ then $\phi(\cl A^{\prime}) \subseteq \cl A^{\prime},$ since for $y \in \cl A^{\prime}, a\phi(y) = \phi(ay) = \phi(ya) = \phi(y)a.$ Thus, if $\phi(x) =y$ then $y \in \cl A^{\prime \prime}$ and hence, $\ell_{\cl A}(x) \le y \le u_{\cl A}(x)$ from which it follows that $y=x.$ Hence, $\{ x \in \cl A^{\prime \prime}_h: \ell_{\cl A}(x) = u_{\cl A}(x) \} \subseteq \cl F(\cl A).$
\end{proof}

We now apply the above result to several concrete cases.
 To this end let $C([0,1])$ denote the continuous functions on [0,1], let $L^{\infty}([0,1])$ denote the set of equivalence classes of essentially bounded Lebesgue measurable functions and regard, $C([0,1]) \subseteq L^{\infty}([0,1])$ by identifying a continuous function with its equivalence class.
Since $L^{\infty}([0,1])$ is an injective vonNeumann algebra, we have copies of $I(C([0,1]))$ embedded completely order isomorphically inside $L^{\infty}([0,1])$ and corresponding minimal (completely) positive projections onto these copies of the injective envelope. As before, we will have that the intersections of the ranges of all these projections is exactly the set of all elements of $L^{\infty}([0,1])$ that are fixed by every positive map that fixes $C([0,1]).$ We let $\cl F(C([0,1]))$ denote this space.

To understand this example, it helps to notice some facts about sup's and inf's in $L^{\infty}([0,1]).$ Let $[g] \in L^{\infty}([0,1]),$ recall that the set of points $x$'s such that
$$\lim_{h \to 0^+} \frac{1}{2h} \int_{x-h}^{x+h} g(t) dm(t)$$
exists, is independent of the particular choice of function from the equivalence class of $g$ and is a set of full measure. These points are called the {\em Lebesgue points} of $g.$ We let $E_g$ denote the set of Lebesgue points of $g$ and we let $\tilde{g}$ be the function whose domain is $E_g$ and which is equal to this limit at each Lebesgue point.

Given $[g] \in L^{\infty}([0,1])$, we define
$$g_l(t) = \sup \{ f(t): f \in C([0,1]), f(x) \le \tilde{g}(x) \forall x \in E_g \}$$
and
$$g_u(t) = \inf \{ f(t): f \in C([0,1]), g(x) \le f(x) \forall x \in E_g \}.$$

Note that $g_l$ is lower semicontinuous, $g_u$ is upper semicontinuous and $g_l(x) \le \tilde{g}(x) \le g_u(x), \forall x \in E_g.$ We should also note that these functions are not the usual upper and lower envelopes of $g$ that one encounters in Riemann integration. The usual lower and upper envelopes, which we will denote $g^l$ and $g^u$ are defined as above with $g$ in the place of $\tilde{g}$ and the inequalities required to hold at all points. For example, if $g$ is the characteristic function of the rationals, then $g^l$ is constantly 0 while $g^u$ is constantly 1, but $E_g=[0,1]$ and $g_l(x)= \tilde{g}(x) = g_u(x)=0.$
Note that for $x \in E_g,$ we have that $f^l(x) \le f_l(x) \le \tilde{g}(x) \le g_u(x) \le g^u(x).$

\begin{prop} Let $[g] \in L^{\infty}([0,1])$, then in the lattice of $L^{\infty}([0,1])$ we have that
$$\ell_{C([0,1])}([g]) = \sup \{ [f]: f \in C([0,1]), [f] \le [g] \} = [g_l]$$
and
$$u_{C([0,1])}([g]) = \inf \{ [f]: f \in C([0,1]), [g] \le [f] \} = [g_u].$$
\end{prop}
\begin{proof}
We only prove the first equality. Let $[h]$ denote the supremum. If $f \in C([0,1])$ and $[f] \le [g]$, then $[f] \le [h]$ and so $f \le g, a.e.$ m and $f \le h, a.e.$ m. Hence, $f(x) \le \tilde{g}(x) \forall x \in E_g$ and so $f(x) \le g_l(x) \forall x \in E_g.$ Since, this set is full measure, $[f] \le [g_l]$ and so $[h] \le [g_l].$ But, we also have that $f(x) \le \tilde{h}(x) \forall x \in E_h.$ Hence, $g_l(x) \le \tilde{h}(x) \forall x \in E_h \cap E_g$ and so, $[g_l] \le [h],$ since $E_h \cap E_g$ is a set of full measure.
\end{proof}

Note that unlike the case of the continuous functions, it is possible for two Riemann integrable functions to be equal almost everywhere without being equal. Thus, the inclusion of the Riemann integrable functions into $L^{\infty}([0,1])$ is not a monomorphism. Moreover, if a function is equal almost everywhere to a Riemann integrable function, it need not be Riemann integrable.

\begin{thm} For $C([0,1]) \subset L^{\infty}([0,1])$, the set $\cl F_{L^{\infty}([0,1])}(C([0,1]))$ is equal to the set of equivalence classes of Riemann integrable functions. That is $[g] \in \cl F(C([0,1]))$ if and only if $f=g, a.e.$ for some Riemann integrable function $f.$
\end{thm}

\begin{proof} By the above theorem, we have that $[g] \in \cl F_{L^{\infty}([0,1])}(C([0,1]))_h$ if and only if $[g_l]=[g]=[g_u].$
Since $g_l(x) \le g_u(x) \forall x \in [0,1]$ and are equal almost everywhere, these functions are both Riemann integrable. Thus, if $[g] \in \cl F_{L^{\infty}([0,1])}(C([0,1]))_h$, then $g = g_l a.e.$ and $g_l$ is Riemann integrable. 

Conversely, if $g=f a.e.$ are real-valued and $f$ is Riemann integrable, then $f^l=f^u a.e.$ and hence from the above inequalities, $f_l=f_u a.e.,$ so that $[g_l]=[f_l]=[f]=[g]=[f_u]=[g_u]$ which implies that $[g] \in \cl F_{L^{\infty}([0,1])}(C([0,1])).$
 
\end{proof}

\begin{cor} Let $\pi:C([0,1]) \to B(L^2([0,1]))$ be the *-monomorphism given by $\pi(f) =M_f,$ where $M_f$ denotes the operator of multiplication by $f$. Then $\cl F( \pi(C([0,1]))) = \{ M_f: f \in \cl R([0,1]) \},$ where $\cl R([0,1])$ denotes the set of Riemann integrable functions and hence is a C*-algebra.
\end{cor}
\begin{proof} We have that $\pi(C([0,1])^{\prime}= \pi(C([0,1]))^{\prime \prime} = \{ M_f: f
  \in L^{\infty}([0,1]) \} \equiv L^{\infty}([0,1])$ and the result follows.

The last statement follows from the fact that the Riemann integrable functions are a C*-algebra and that $g \to M_g$ is a *-homomorphism.
\end{proof}

We now consider a discrete case. Let $X$ be a compact, Hausdorff space, let $\ell^{\infty}(X)$ denote the bounded functions on $X$ and let $C(X),$  $LSC(X),$ and $USC(X)$ denote the continuous, lower semicontinuous and upper semicontinuous functions on $X$, respectively. Also, let $\{ x_i \}_{i \in I}$ be a dense set in $X$. Recall that a function is lower semicontinuous(respectively, upper semicontinuous) if and only if it is the supremum(respectively, infimum) of the continuous functions that are less(repsectively, greater) than it.

\begin{prop} Let $\pi:\ell^{\infty}(X) \to \ell^{\infty}(I)$ be defined by $\pi(f)(i)= f(x_i).$ Then
$$\cl F_{\ell^{\infty}(I)}(\pi(C(X))) = \pi(LSC(X)) \cap \pi(USC(X)).$$
\end{prop}
\begin{proof}
We have that $h \in \cl F_{\ell^{\infty}(I)}(\pi(C(X)))$ if and only if $h= \sup \{ \pi(f): f \in C(X), \pi(f) \le h \} = \inf \{ \pi(f): f \in C(X), h \le \pi(f) \}.$

Let $g_l(x) = \sup \{ f(x): f \in C(X), \pi(f) \le h \}$ and $g_u(x) = \inf \{ f(x): f \in C(X), \pi(f) \le h \}.$ Then $g_l$ is lower semicontinuous and $g_u$ is upper semicontinuous and $h = \pi(g_l) = \pi(g_u).$
Hence, $\cl F_{\ell^{\infty}(I)}(\pi(C(X))) \subseteq \pi(LSC(X)) \cap \pi(USC(X)).$

Conversely, if $h \in \pi(LSC(X)) \cap \pi(USC(X)),$ say $h = \pi(g_l)=\pi(g_u).$ Note that for $f \in C(X),$ we have that  $\pi(f) \le h,$ if and only if $f \le g_l.$
Hence, $\sup \{ \pi(f): f \in C(X), \pi(f) \le h \} = \sup \{ \pi(f): f \in C(X), f \le g_l \} = \pi(g_l) = h.$ similarly, $h = \inf \{ \pi(f): f \in C(X), h \le \pi(f) \}$ and the result follows.
\end{proof}

\begin{cor} Let $X$ be a compact, Hausdorff space, let $\{ x_i \}_{i \in I}$ be a dense set of distinct points in $X$, let $\{ e_i \}$ denote the canonical orthonormal basis for $\ell^2(I),$ and let $\pi: \ell^{\infty}(X) \to B(\ell^2(I))$ be the diagonal representation defined by $\pi(f)e_i= f(x_i)e_i \forall i.$  Then $\cl F(\pi(C(X))) = \pi(LSC(X)) \cap \pi(USC(X)).$
\end{cor}
\begin{proof} The result follows as above, since $\pi(C(X))^{\prime \prime} = \pi(C(X))^{\prime}.$
\end{proof}

Consider the case of $X=[0,1],$ with a dense subset given by an enumeration of the rationals, $\{ r_n \}_{n \in \bb N},$ and $\pi: \ell^{\infty}([0,1]) \to B(\ell^2(\bb N)),$ given by the above formula. If we consider an interval with irrational endpoints, $a$ and $b$, then $\pi(\chi_{[a,b]}) = \pi(\chi_{(a,b)})$ and so this projection belongs to $\cl F(\pi(C([0,1]))).$
However, it can be seen that no finite rank diagonal projection or a projection corresponding to an interval with a rational endpoint belongs to $\cl F(\pi(C([0,1])).$

\section*{Acknowledgments}

The author wishes to thank Ken Davidson, Roger Smith and Ivan Todorov,
for various conservations that contributed to this work.


\begin{thebibliography}{99}

\bibitem{Ar}  W. B. Arveson, {\em An Invitation to C*-algebras,}
  Graduate Texts in Mathematics, vol. 39,
  Springer-Verlag, New York, 1976.

\bibitem{Bl}  B. E. Blackadar, {\em Weak Expectations and Nuclear C*-algebras,} Indiana University Mathematics J., Vol. 27, No. 6(1978), 1021-1026.



\bibitem{BLeM} D. P. Blecher and C. LeMerdy, {\em Operator Algebras and
    Their Modules,} London Mathematical Society Monographs, vol. 30,
    Oxford Science Publications, Oxford, England, 2004.


\bibitem{Ch} M. D. Choi, {\em A Schwarz inequality for positive linear
    maps on C*-algebras,} Illinois J. Math. {\bf 18}(1974), 565--574.







\bibitem{H1} M. Hamana, {\em Injective envelopes of C*-algebras,}
J. Math. Soc. Japan Vol. 31, No. 1, (1979), 181-197.

\bibitem{H2}  M. Hamana, {\em Injective
envelopes of operator systems,}  Publ. Res. Inst. Math. Sci.  15(1979), no. 3, 773--785.

\bibitem{H4} M. Hamana, {\em Regular embeddings of $C^*$-algebras in monotone complete $C^*$-algebras,}  J. Math. Soc. Japan  33(1981), no. 1, 159--183. 


\bibitem{KR} R. V. Kadison and J. R. Ringrose, {\em Fundamentals 
of
    the theory of operator algebras Vol. 1: Elemental Theory},
    Academic Press, New York (1983). The current publication is 
from
    Amer. Math. Soc., Providence, RI, (1997) with the same title 
and contents.



\bibitem{Pi} G. Pisier, {\em Introduction to Operator Space Theory,}
London Mathematical Society Lecture Note Series {\bf 294}, Cambridge
University Press, Cambridge, United Kingdom, 2003.



\bibitem{Ru} Z.-J. Ruan, {\em Injectivity of Operator Spaces,}
  Trans. Amer. Math. Soc. {\bf315}(1989), 89-104.






\bibitem{MW1} J. D. Maitland Wright, {\em Every monotone $\sigma$-complete C*-algebra is the quotient of its Baire* envelope by a two sided $\sigma$-ideal,} J. London Math. Soc. (2), {\bf 6}(1973), 210--214.

\bibitem{MW2} J. D. Maitland Wright, {\em On minimal $\sigma$-completions of C*-algebras,} Bull. London Math. Soc., {\bf 6}(1974), 168--174.

\bibitem{MW3} J. D. Maitland Wright, {\em Regular $\sigma$-completions of C*-algebras,} J. London Math. Soc. (2), {\bf 12}(1976), 299--309.

\end{thebibliography}
\end{document}